\documentclass[a4paper,11pt,twoside]{amsart}

\usepackage[english]{babel}
\usepackage[utf8]{inputenc}
\usepackage[a4paper,inner=3cm,outer=3cm,top=4cm,bottom=4cm,pdftex]{geometry}
\usepackage{fancyhdr}
\usepackage{titlesec}
\usepackage{enumerate}
\usepackage{color}
\usepackage{bold-extra}
\usepackage{graphics}
\usepackage{aliascnt}
\usepackage{cite}
\usepackage[pdftex,citecolor=green,linkcolor=red]{hyperref}
\usepackage{tikz}
\usepackage{amsmath}
\usepackage{amsfonts}
\usepackage{amssymb}
\usepackage{amsthm}
\usepackage{mathtools}
\usepackage{tikz-cd}
\usepackage{parskip}
\usepackage{cancel}
\usepackage{cite}
\usepackage{soul}

\titleformat{\section}{\normalfont\scshape\centering}{\thesection}{1em}{}
\titleformat{\subsection}{\bfseries}{\thesubsection}{1em}{}
  
\theoremstyle{plain}
\newtheorem{theorem}{Theorem}[section]
\newtheorem*{main}{Main Theorem}
\newtheorem*{prob}{Index Map Problem}
\newtheorem{corollary}[theorem]{Corollary}

\newtheorem{lemma}[theorem]{Lemma}
\newtheorem{proposition}[theorem]{Proposition}
\theoremstyle{definition}

\newtheorem{remark}[theorem]{Remark}

\newtheorem{example}[theorem]{Example}

\newcommand{\Ind}{\textup{Ind}}
\newcommand{\Gal}{\textup{Gal}}

\newcommand{\lcm}{\textup{lcm}}
\newcommand{\im}{\textup{Im }}

\newcommand{\rank}{\textup{rank }}
\newcommand{\p}{\mathfrak{p}}

\newcommand{\mat}{\textup{Mat}}

\DeclareMathOperator{\tors}{tors}

\title{Unified treatment of Artin-type problems
}
\author{Olli J\"arviniemi and Antonella Perucca}
\address[]{Department of Mathematics and Statistics, University of Turku, 20014 Turku, Finland}
\email{olli.a.jarviniemi@utu.fi}
\address[]{Department of Mathematics, University of Luxembourg, 6 av.\@ de la Fonte, 4364 Esch-sur-Alzette, Luxembourg}
\email{antonella.perucca@uni.lu}


\begin{document}

\maketitle

\begin{abstract}
Since Hooley's seminal 1967 resolution of Artin's primitive root conjecture under the Generalized Riemann Hypothesis, numerous variations of the conjecture have been considered. We present a framework generalizing and unifying many previously considered variants, and prove results in this full generality (under GRH). 
\end{abstract}

\section{Introduction}
Several problems related to Artin's primitive root conjecture may be viewed as instances of the following:

\begin{prob}\label{BIGproblem}
Let $K$ be a number field. Let $W_1, \ldots , W_n$ be finitely generated subgroups of $K^{\times}$ of positive rank. For all but finitely many primes $\mathfrak{p}$ of $K$, their reductions modulo $\mathfrak p$ are subgroups of $k_\p^\times$ of finite index ($k_\p$ being the residue field), yielding the \emph{index map}
$$\mathfrak{p}\mapsto (\Ind_{\mathfrak{p}}(W_1), \ldots , \Ind_{\mathfrak{p}}(W_n))\,.$$
Determine whether the preimage of a subset of $\mathbb{Z}_{>0}^n$ is infinite (respectively, it has a positive natural density). Possibly consider a finite Galois extension $F/K$, a union of conjugacy classes $C \subset \Gal(F/K)$, and restrict the index map to the primes $\p$ unramified in $F$ such that $\left(\frac{F/K}{\mathfrak{p}}\right) \subset C$.
\end{prob}

Consider the preimages of a tuple $(h_1, \ldots , h_n) \in \mathbb{Z}_{>0}^n$. The original Artin's conjecture (possibly for higher rank) corresponds to the case $n = 1$ and $h_1=1$ because we look for a \emph{primitive root}, and for a \emph{near-primitive root} we can vary $h_1\in \mathbb{Z}_{>0}$. For \emph{simultaneous primitive roots} we may consider the preimage of the tuple $(1, \ldots , 1)$, and more generally the \emph{Schinzel-Wójcik problem} concerns the constant tuples. For the \emph{two variable Artin conjecture} we consider $n=2$ and all pairs $(h_1,h_2)$ such that $h_1\mid h_2$. For the \emph{smallest primitive root problem} we let  $W_1,\ldots, W_n$ be generated by the first $n$ integers greater than one and consider the tuples such that $h_n$ is the only entry equal to $1$. Notice that, over $\mathbb Q$, for all above-mentioned problems we could restrict to primes satisfying some congruence condition, which can be expressed by a suitable choice of $F$ and $C$. Also notice that, in case we do not need the Frobenius condition, we may simply take $F=K$. The above list of questions is surely non-exhaustive: we refer the reader to the survey by Moree on Artin's Conjecture \cite{m-survey} and, to name a few, to the following papers \cite{hooley, heath-brown, pappalardi, matthews, wojcik, ms, mss}.

The main results of this work address the Index Map Problem in full generality.

\begin{main}
Assume GRH. Fix a finite Galois extension $F/K$ and a union of conjugacy classes $C \subset \Gal(F/K)$. Restrict the index map to the primes $\mathfrak{p}$ of $K$ unramified in $F$ such that $\left(\frac{F/K}{\mathfrak p}\right) \subset C$.
\begin{enumerate}
\item[(i)] [Theorem~\ref{MasterTheorem} and Remark~\ref{masterremark}] The preimage under the index map of any non-empty subset of $\mathbb{Z}_{>0}^n$ is either finite or it has a positive natural density.
\item[(ii)] [Theorems \ref{wishlist}] The image of the index map is computable with an explicit finite procedure.

\item[(iii)] [Theorem~\ref{thm-imf-sep}] If $F=K$, then the following two conditions are equivalent: 
\begin{itemize}
\item[-] the image of the index map contains all positive multiples of some tuple $(H_1,\ldots, H_n)$;
\item[-] for every $i=1,\ldots, n$ the rank of $\langle W_1,\ldots, W_n\rangle$ is strictly larger than the rank of $\langle W_1,\ldots, \cancel{W_i}, \ldots, W_n\rangle$.
\end{itemize}
Moreover, if the above conditions hold, then a tuple $(h_1,\ldots, h_n)\in \mathbb{Z}_{>0}^n$ is in the image of the index map if and only if the same holds for the tuple $$(\gcd(h_1, H_1),\ldots, \gcd(h_n, H_n))\,.$$
\end{enumerate}
\end{main}

Let $F=K$. In Theorem~\ref{image-f-Q} we explicitly determine the image of the index map when $K = \mathbb{Q}, n = 1$ and $W_1$ has rank $1$. While for $K = \mathbb{Q}$ the index map is never surjective onto $\mathbb{Z}_{>0}^n$ if $n \geqslant 2$, we prove that for any $K\neq \mathbb Q$ and for any $n$ there are groups $W_1,\ldots, W_n$ for which the index map is surjective onto $\mathbb{Z}_{>0}^n$, see Section~\ref{surjective}. Notice that in the Index Map Problem the assumption that the groups have positive rank is justified by Remarks~\ref{remarkNEW1} and \ref{remarkROOTS}.

We make use of several results of Kummer theory, which we collect in Section~\ref{Kummer} (some of our statements may be new). We also prove the following result of Kummer theory, that holds for every number field $K\neq \mathbb Q$: there are countably many elements of $K^\times$ such that any finite subset of them gives rise to Kummer extensions with maximal degree (in short, for  $\alpha_1,\ldots, \alpha_r\in K^\times$ , the degree of $K(\zeta_{mn}, \alpha_1^{1/n}, \ldots, \alpha_r^{1/n})$ over $K(\zeta_{mn})$ equals $n^r$ for all positive integers $n$ and $m$), see Proposition~\ref{infinitelymany}.

Our main technical tool is Theorem~\ref{MasterTheorem}, which is a generalization of a celebrated result by Lenstra \cite{lenstra77}, see also \cite[Section 4]{jarviniemi} (notice that the analytic ideas remain unchanged from Hooley \cite{hooley}). Since in the proof of Theorem~\ref{MasterTheorem} we use an effective version of the Chebotarev density theorem which is conditional on GRH, most of our results on the Index Map Problem are conditional on GRH, as it is customary for problems related to Artin's primitive root conjecture.

\section{Notations for the paper}\label{Notation}

{\bf General notation.} In this paper $K$ is a number field, $F/K$ a finite Galois extension of $K$ with Galois group $\Gal(F/K)$, and $C \subset \Gal(F/K)$ is a union of conjugacy classes. 
The multiplicative group of $K$ is denoted by $K^{\times}$, the ring of integers by $O_K$, and we fix some algebraic closure $\bar{K}$ containing $F$. We call $\tau$ the order of the torsion subgroup of $K^\times$. If $n$ is a positive integer, then we write $\zeta_n$ for a root of unity of order $n$ (such that for all positive integers $m$ we have $\zeta_{mn}^m=\zeta_n$). We also use the notation $K(\zeta_{\infty})$ and $K(\zeta_{\ell^\infty})$, in the obvious way, to define infinite cyclotomic extensions of $K$. We call $\Omega$ the smallest positive integer such that $K\cap \mathbb Q(\zeta_\infty)\subset \mathbb Q(\zeta_\Omega)$.
The letters $p, q, \ell$ are reserved for primes (of $\mathbb{Z}$), and $\mathfrak{p}, \mathfrak{q}$ for primes of $K$ i.e.\ non-zero prime ideals of $O_K$. The ideal norm of $\mathfrak{p}$ is $N\mathfrak{p}$, the residue field is $k_\mathfrak p$, and the Artin symbol is the conjugacy class  $\left(\frac{F/K}{\mathfrak{p}}\right)$.

{\bf Tuples.} We write $I=\{1,\ldots, n\}$ and hence we denote e.g.\ by $h_I$ the tuple $(h_1,\ldots, h_n)$. 
If $h_I$ and $H_I$ are in $\mathbb Z_{\geqslant 0}^n$, then $h_I$ is smaller than or equal to $H_I$ if $h_i\leqslant H_i$ holds for every $i\in I$, and we define $\min(h_I, H_I)$ as the tuple with entries $\min(h_i, H_i)$. Moreover, we call $\oplus h_I$ the set consisting of the tuples larger than or equal to $h_I$.
If $h_I$ and $H_I$ are in $\mathbb Z_{> 0}^n$, then $h_I$ divides $H_I$ if $h_i\mid H_i$ holds for every $i\in I$, and $\gcd(h_I, H_I)$ is the tuple with entries $\gcd(h_i, H_i)$.
Moreover, we call $\otimes h_I$ the set consisting of the multiples of $h_I$, and the $\ell$-adic valuation $v_\ell(h_I)$ is the tuple with entries $v_\ell(h_i)$.

{\bf Groups.}
The symbol $W$ stands for a finitely generated subgroup of $K^\times$ of positive rank (by which we mean the rank of $W$ modulo its torsion subgroup $W_{\tors}$). For all but finitely many $\mathfrak{p}$ the reduction $(W \bmod \mathfrak{p})$ is a well-defined subgroup of $k_\mathfrak p^\times$ and we consider its index $\Ind_{\mathfrak{p}}W$, which we also write as $\Ind_{\mathfrak{p}}\alpha$ if $W=\langle \alpha \rangle$.
If $n\in \mathbb Z_{>0}$, then $W^{1/n}\subset \bar{K}^\times$ consists of all $n$-th roots of all elements of $W$ (we also consider $S^{1/n}$ for $S\subset K^\times$), and we similarly define $W^{1/\infty}$ and $W^{1/{\ell^\infty}}$.

Let $I=\{1,\ldots, n\}$ and for $i\in I$ consider subgroups $W_i$ of $K^\times$ of positive rank. Call $W:=\langle W_i \mid i\in I\rangle$ and $W_{\neq i}:=\langle W_j \mid j\in I, j\neq i \rangle$.

{\bf Separated groups.} We call $W_1,\ldots, W_n$ \emph{separated} if for all $i\in I$ we have 
$$\rank W > \rank W_{\neq i}\,.$$

{\bf The index map.}
For all but finitely many primes $\mathfrak p$ of $K$ we can define the \emph{index map}
$$\Psi: \mathfrak{p}\mapsto (\Ind_{\mathfrak{p}}(W_1), \ldots , \Ind_{\mathfrak{p}}(W_n))\in \mathbb{Z}_{>0}^n \,.$$
Its $\ell$-adic valuation $\Psi_\ell$ is the composition of $\Psi$
with the $\ell$-adic valuation $\mathbb{Z}_{>0}^n\rightarrow \mathbb{Z}_{\geqslant 0}^n$.
We similarly define its $Q$-adic valuation $\Psi_Q$ for any positive squarefree integer $Q$, where the $Q$-adic valuation $v_Q(z)$ of $z\in \mathbb Z_{\neq 0}$ is the tuple consisting of the $\ell$-adic valuation $v_\ell(z)$ for the primes $\ell\mid Q$.

{\bf Convention.} Starting from Theorem~\ref{MasterTheorem}, up to excluding finitely many primes of $K$, we restrict $\Psi$ to the primes that are not ramified over $\mathbb Q$, do not ramify in $F$, and do not lie over a prime divisor of $\tau$.

\section{Results from Kummer theory}\label{Kummer}

\begin{proposition}\label{justW}
The following holds:
\begin{enumerate}
\item[(i)] There exists a positive integer $z$ (depending on $K$ and $W$, computable with an explicit finite procedure) such that for all positive integers $t$ we have
$$[K(\zeta_\infty, W^{1/zt}) : K(\zeta_\infty, W^{1/z})]=t^{\rank W}\,.$$
\item[(ii)] There exists a finite Galois extension $K_W/K$  (depending on $K$ and $W$, computable with an explicit finite procedure) such that for all coprime positive integers $m, t$ we have 
$$K\left(\zeta_{m}, W^{1/m}\right) \cap K\left(\zeta_{t}, W^{1/t}\right) \subset K_W\,.$$
\item[(iii)] There exists a positive integer $c$  (depending on $K$ and $W$, computable with an explicit finite procedure) such that for all positive integers $m,t$ such that $m$ is coprime to both $t$ and $c$ we have
$$ K\left(\zeta_m, W^{1/m}\right)\cap K\left(\zeta_{t}, W^{1/t}\right)=K$$
$$ \text{and}\quad [K\left(\zeta_m, W^{1/m}\right) : K]=\varphi(m)m^{\rank W}\,.$$
We could additionally require that $K\left(\zeta_m, W^{1/m}\right)\cap F=K$ (in this case $c$ also depends on $F$). We could additionally require that $c$ is divisible by some given positive integer, or we may replace $c$ by its squarefree part.
\end{enumerate}
\end{proposition}

\begin{proof}
For (i) we may suppose that $W$ is torsion-free. By \cite[Theorems 1.1 and 1.2]{perucca-sgobba-tronto} we can take $z$ to be a positive integer maximizing the integer ratio
$$\frac{\varphi(z) z^{\rank W}}{[K\left(\zeta_z, W^{1/z}\right) : K]}\,.$$
We may also suppose that $K\cap \mathbb Q(\zeta_\infty)\subset K(\zeta_z)$, so that for (ii) and (iii) we may take $K_W=K\left(\zeta_z, W^{1/z}\right)$ and $c=z$. For the additional remarks in (iii) notice that, since we may replace $c$ by a multiple, we may suppose $\Omega\mid c$ and $F\cap K\left(\zeta_\infty, W^{1/\infty}\right)\subset K\left(\zeta_c, W^{1/c}\right)$.
\end{proof}

In the following result we introduce the terminology of \emph{separated} (respectively, \emph{$\ell$-separated}) \emph{Kummer extensions}.

\begin{proposition}\label{used}
Suppose that $W_1,\ldots, W_n$ are separated.
\begin{enumerate}
\item[(i)] There exists $h_I\in \mathbb Z_{>0}^n$ such that the $W_i$'s have \emph{separated Kummer extensions} w.r.t.\ $h_I$. By this we mean that the fields $$K(\zeta_\infty, W_1^{1/H_1}, \ldots, W_n^{1/H_n}) \qquad \text{for}\qquad H_I\in \otimes h_I$$ are pairwise distinct. Equivalently, by varying $i \in I$ and $H_i\in h_i \mathbb Z\cap \mathbb Z_{>0}$, the fields $K(\zeta_\infty, W_i^{1/H_i}, W_{\neq i}^{1/\infty})$ are pairwise distinct.
\item[(ii)] For every $\ell$ there exists $e_I\in \mathbb Z_{\geqslant 0}^n$ such that the $W_i$'s have \emph{$\ell$-separated Kummer extensions} w.r.t.\ $e_I$. By this we mean that the fields
$$K(\zeta_{\ell^\infty}, W_1^{1/\ell^{E_1}}, \ldots, W_n^{1/\ell^{E_n}}) \qquad \text{for}\qquad E_I\in \oplus e_I$$
 are pairwise distinct. Equivalently, by varying $i \in I$ and $E_i\in \mathbb Z_{\geqslant e_i}$, the fields $K(\zeta_{\ell^\infty}, W_i^{1/\ell^{E_i}}, W_{\neq i}^{1/\ell^{\infty}})$ are pairwise distinct.
\item[(iii)] Consider the set $Sep\subset \mathbb Z_{>0}^n$  (respectively, $Sep_\ell \subseteq \mathbb Z_{\geqslant 0}^n$) consisting of the tuples with respect to which the $W_i$'s  have separated (respectively, $\ell$-separated) Kummer extensions. The set $Sep$ and all sets $Sep_\ell$ are computable at once with an explicit finite procedure.
There exist $h_I\in \mathbb Z_{>0}^n$ and $e_{\ell,I}\in \mathbb Z_{\geqslant 0}^n$ such that $Sep=\otimes h_I$, $Sep_\ell=\oplus e_{\ell,I}$, and for every $i\in I$ we have $0\leqslant v_\ell(h_i)- e_{\ell, i} \leqslant v_\ell (\tau)$ and 
\begin{align*}
v_\ell(h_i)&= \max \{ z\,: \; K(\zeta_{\infty}, W_{i}^{1/\ell^{z}})\subset K(\zeta_{\infty}, W_{\neq i}^{1/\ell^{\infty}}, W_i^{1/\ell^{z-1}})\}\\ 
e_{\ell, i}&= \max \{ z\,: \; K(\zeta_{\ell^\infty}, W_{i}^{1/\ell^{z}})\subset K(\zeta_{\ell^\infty}, W_{\neq i}^{1/\ell^{\infty}}, W_i^{1/\ell^{z-1}})\}\,.
\end{align*}
\end{enumerate}
\end{proposition}

By Proposition~\ref{used}(i), for every $H_I\in \otimes h_I$ the degree $$[K(\zeta_\infty, W_1^{1/H_1}, \ldots, W_n^{1/H_n}) : K(\zeta_\infty, W_1^{1/h_1}, \ldots, W_n^{1/h_n})]$$
is a multiple of $\prod_{i\in I} H_i/h_i$ (stepwise, multiplying one parameter by a prime $\ell$ gives a non-trivial Kummer extension whose degree is a multiple of $\ell$).
Similarly, by Proposition~\ref{used}(ii) for every $E_I\in \oplus e_I$ the $\ell$-adic valuation of the degree 
$$[K(\zeta_{\ell^\infty}, W_1^{1/\ell^{E_1}}, \ldots, W_n^{1/\ell^{E_n}}) : K(\zeta_{\ell^\infty}, W_1^{1/\ell^{e_1}}, \ldots, W_n^{1/\ell^{e_n}})]$$
is at least $\sum_{i\in I} (E_i-e_i)$.

\begin{proof}
To prove the first equivalence (the second is analogous), consider tuples $H_I, H_I'\in \otimes h_I$. If we have $K(\zeta_\infty, W_1^{1/H_1}, \ldots, W_n^{1/H_n})=K(\zeta_\infty, W_1^{1/H'_1}, \ldots, W_n^{1/H'_n})$, then clearly for every $i$ we have $K(\zeta_\infty, W_i^{1/H_i}, W_{\neq i}^{1/\infty})=K(\zeta_\infty, W_i^{1/H'_i}, W_{\neq i}^{1/\infty})$. For the converse: 
if the latter equality holds for some $i$ and for some $H_i, H_i'$, then for $j\neq i$ we can set $H_j=H_j'=Z$ for some (sufficiently divisible) positive integer $Z$ so that the former equality holds.

Now we prove (i), the proof of (ii) being analogous. Let $z$ be as in Proposition~\ref{justW}(i) for $W$, thus for every positive multiple $Z$ of $z$ we have
$$[K(\zeta_\infty, W^{1/Z}): K(\zeta_\infty, W^{1/z})]=(Z/z)^{\rank W}\qquad \text{and}$$
$$[K(\zeta_\infty, W_{\neq i}^{1/Z}): K(\zeta_\infty, W_{\neq i}^{1/z})]=(Z/z)^{\rank W_{\neq i}}\,.$$
By assumption $\rank W>\rank W_{\neq i}$ hence for every distinct positive divisors $t,t'$ of $Z/z$ we have 
 $K(\zeta_\infty, W_i^{1/tz}, W_{\neq i}^{1/Z})\neq K(\zeta_\infty, W_i^{1/t'z}, W_{\neq i}^{1/Z})$ 
Indeed, we have 
$$[K(\zeta_{\infty}, W_i^{1/tz}, W_{\neq i}^{1/\infty}) : K(\zeta_{\infty}, W_i^{1/z}, W_{\neq i}^{1/\infty})] = t^{\rank W - \rank W_{\neq i}}\,.$$
As we can freely choose $Z$, we may replace it by $\infty$.
 
We are left to prove (iii). Notice that $m_I \in Sep$ (respectively, $Sep_\ell$) implies $\otimes m_I\subset Sep$ (respectively, $\oplus m_I\subset Sep_\ell$). Call $v_{\ell, i}$ and $e_{\ell, i}$ the two maxima from (iii), and set $h_i:=\prod_\ell \ell^{v_{\ell, i}}$, thus $Sep \subset \otimes h_I$ and $Sep_\ell \subset \oplus e_{\ell, I}$. These inclusions are equalities because by Kummer theory $h_I\in Sep$ and $e_{\ell, I}\in Sep_\ell$. Clearly $v_\ell(h_i)- e_{\ell, i}$ is non-negative, and it is at most $v_\ell (\tau)$ by  Schinzel's Theorem for abelian radical extensions \cite[Theorem 2]{Schinzel} (with the notation of \cite{perucca-sgobba-tronto}, we have to consider the $\ell$-adelic failure for $W$).

By the theory developed in \cite{debry-perucca, perucca-sgobba} we may compute $e_{\ell, I}$ for all $\ell$, which is the zero tuple except for finitely many  computable primes $\ell$ . We are left to compute $v_\ell(h_I)$ for $\ell\mid \tau$. Set $t_\ell:=v_\ell(\tau)+\max_i(e_{\ell, i})$, and let $T_\ell$ be the product of $2\ell^{v_\ell(\tau)}$ and the primes $q\equiv 1 \pmod \ell$ such that $\ell\nmid v_\mathfrak q (w)$ holds for some $w\in W$ and for some prime $\mathfrak q$ of $K$ lying over $q$. 
We conclude because by \cite[Lemma C.1.7 and its proof]{hindry-silverman} we have 
$$v_\ell(h_i)-e_{\ell, i}= v_\ell \left( \frac{[K(\zeta_{T_\ell \ell^{t_\ell}}, W^{1/\ell^{t_\ell}}):K(\zeta_{T_\ell \ell^{t_\ell}}, W_{\neq i}^{1/{\ell^{t_\ell}}})]}
{[K(\zeta_{\ell^{t_\ell}}, W^{1/\ell^{t_\ell}}):K(\zeta_{\ell^{t_\ell}}, W_{\neq i}^{1/{\ell^{t_\ell}}})]}\right)
\,.$$
\end{proof}

With the notation of Proposition~\ref{used} we could have $e_{\ell, i}< v_\ell(h_i)$: indeed, for $K=\mathbb Q$ and $W=\langle 5 \rangle$ we have $Sep_2=\mathbb Z_{\geqslant 0}$ and $Sep=2\, \mathbb Z_{>0}$.

\begin{lemma}\label{lemma:kummer-galois}
Let $\alpha_1, \ldots , \alpha_r \in K^{\times}$ be multiplicatively independent. There exist a positive integer $N$ and $S \subset (\mathbb{Z}/N\mathbb{Z})^{r+1}$ such that for $(z_0, \ldots , z_r)\in \mathbb{Z}_{>0}^{r+1}$ with $\lcm(z_0, \ldots , z_r)=  z_0$ the following holds: given $(x_0, \ldots , x_r)\in \mathbb{Z}_{>0}^{r+1}$, there is 
$$\sigma\in \Gal\left(F\left(\zeta_{z_0}, \alpha_1^{1/z_1}, \ldots , \alpha_r^{1/z_r}\right)/K\right)$$
such that $\sigma|_F \in C$, $\sigma\left(\zeta_{z_0}\right) = \zeta_{z_0}^{x_0}$, and $\sigma\left(\alpha_i^{1/z_i}\right) = \zeta_{z_i}^{x_i}\alpha_i^{1/z_i}$ for all $i=1,\ldots, r$ if and only if $\gcd(x_0, z_0) = 1$ and $(x_0 \bmod{N}, \ldots , x_r \bmod{N}) \in S$. The integer $N$ only depends on $K,F$ and $\langle \alpha_1, \ldots , \alpha_r\rangle$, and $N,S$ are computable with an explicit finite procedure. Supposing that the prime divisors of $z_0, \ldots, z_r$ belong to some set $D$, we may remove from $N$ all prime factors not in $D$.
\end{lemma}
\begin{proof}
The condition $\gcd(x_0, z_0) = 1$ is clearly necessary, so assume that it holds. 
Calling $W=\langle \alpha_1, \ldots , \alpha_r\rangle$, by Proposition \ref{justW} there is some positive integer $N$ such that $F\cap K(\zeta_\infty, W^{1/\infty})\subset K(\zeta_N, W^{1/N})$ and such that for every positive multiple $M$ of $N$ we have 
$$[K(\zeta_M, W^{1/M}): K(\zeta_N, W^{1/N})]=(M/N)^r \cdot \varphi(M)/\varphi(N)\,.$$
Thus we may replace $z_i$ by $\gcd(z_i,N)$ and work in 
$\Gal(K(\zeta_N, W^{1/N})/K)$: if one finds a suitable automorphism of $K(\zeta_N, W^{1/N})$, then one may extend it freely to $K(\zeta_M, W^{1/M})$ as the degree of the extension above is maximal.
\end{proof}

\begin{remark}\label{Debry}
By \cite[Definitions 5 and 10]{debry-perucca}, elements $\alpha_1,\ldots, \alpha_r\in K^\times$ must be strongly $\ell$-independent if for all $i=1,\ldots, r$ there is some prime $\mathfrak q_i$ of $K$ such that $\ell \nmid v_{\mathfrak q_i}(\alpha_i)$ and $\ell \mid v_{\mathfrak q_i}(\alpha_j)$ for all $j\neq i$. Indeed, if $\alpha:=\prod \alpha_{i}^{z_i}$ holds for some integers $z_i$ and w.l.o.g.\ $\ell\nmid z_1$, then $\ell \nmid v_{\mathfrak q_1}(\alpha)$ hence $\alpha$ is strongly $\ell$-indivisible. Similarly, if the $\alpha_i$'s are not units and their norms $N_{K/\mathbb Q}(\alpha_i)$ are pairwise coprime and not of the form $\pm 1$ times an $\ell$th power, then they are an $\ell$-good basis (as in \cite[Theorem 14]{debry-perucca}) for $\langle \alpha_1,\ldots, \alpha_r \rangle$.
\end{remark}

\subsection*{Kummer extensions contained in cyclotomic extensions}

Consider a prime power $\ell^e\mid \tau$ and a cyclic Kummer extension $L/K$ of degree $\ell^e$ with $L\subset K(\zeta_\infty)$. Call $K'$ the largest subfield of $K\cap \mathbb Q(\zeta_\infty)$ of degree a power of $\ell$ (over $\mathbb{Q}$). We then have $L=L'K$ for some unique $K'\subset L'\subset \mathbb Q(\zeta_\infty)$ such that $[L':\mathbb Q]$ is a power of $\ell$ (and hence $[L': K']=\ell^e$).
We have $L' \subset \mathbb Q(\zeta_n)$ for some positive integer $n$ of the form 
\begin{align*}
n = \ell^{T+e}Q\qquad\text{where}\qquad  Q=\prod_{q \equiv 1\!\!\!\!\! \pmod{\ell}} q\qquad\text{and}\qquad \zeta_{\ell^T}\in K(\zeta_Q)\,,
\end{align*}
where $q$ denotes a prime. Indeed, if $L' \subset \mathbb{Q}(\zeta_n)$ and $q$ is a prime divisor of $n$ with $q \not\equiv 0, 1 \pmod{\ell}$, then $L' \subset \mathbb{Q}(\zeta_{n/q})$ and hence we may  replace $n$ by $n/q$; the same holds if $q \equiv 1 \pmod{\ell}$ and $q^2 \mid n$. Notice that the extension $L'/K'$ corresponds to a cyclic quotient of order $\ell^e$ of $G:=\Gal(\mathbb Q(\zeta_n)/K')$.

\begin{proposition}\label{gq}
Replace $G$ by its largest quotient of exponent $\ell^e$ and suppose that $G \simeq \prod_{i} \mathbb Z/\ell^{e_i}\mathbb Z$. The $i$-th factor corresponds to a cyclic Kummer extension $K'(\gamma_i)/K'$ of degree $\ell^{e_i}$ for some $\gamma_i\in \mathbb Q(\zeta_n)$ such that $\gamma_i^{\ell^{e_i}}\in K'$. 
\begin{enumerate}
\item[(i)] We have $L'=K'(\gamma)$ for some $\gamma:=\prod_{i} \gamma_i^{y_i}$, where $0\leqslant y_i < \ell^{e_i}$ and for some $i$ we have $e_i=e$ and $\ell\nmid y_i$.

\item[(ii)] For every $q\mid n$ such that $q\nmid \ell\Omega$ fix an algebraic integer $g_q\in \mathbb Q(\zeta_{\ell^e q})$ such that $g_q^{\ell ^e}\in \mathbb Q(\zeta_{\ell^e})$ and such that 
$$v_\ell ([\mathbb Q(\zeta_{\ell^e}, g_q):\mathbb Q(\zeta_{\ell^e})])=\varepsilon_q\qquad\text{where}\quad  \varepsilon_q:=\min (e, v_\ell(q-1))\,.$$
Notice that $g_q$ depends on $q,\ell^e$ but not on $K$ and $L$.
For every $q\mid n$ we may fix $\gamma_i$ as follows: if $q\nmid \ell\Omega$, then $\gamma_i^{\ell^{e_i}}=g_q^{\ell^{\varepsilon_q}}$, else $\gamma_i$ is an algebraic integer in $\mathbb Q(\zeta_{\ell^e\Omega})$. Then $\ell\nmid v_\mathfrak p (\gamma_i^{\ell^{e_i}})$ holds for a prime $\mathfrak p$ of $K$ only if $\mathfrak p$ lies over $q$ or over a divisor of $\ell\Omega$ respectively, and in the former case there is such $\mathfrak p$.
\end{enumerate}
\end{proposition}
\begin{proof}
Since the extensions $K'(\gamma_i)/K'$ are linearly disjoint, 
$\prod_{i} \gamma_i^{z_i}\in K'$ implies $\ell^{e_i}\mid z_i$ for every $i$. Let $0\leqslant y_i,Y_i < \ell^{e_i}$ and suppose that for some $i_1,i_2$ we have $e_{i_1}=e_{i_2}=e$ and $\ell\nmid y_{i_1}, Y_{i_2}$. Then $\prod_{i} \gamma_i^{y_i}$ and $\prod_{i} \gamma_i^{Y_i}$ generate the same extension if and only if $Y_i\equiv ty_i \pmod {\ell^{e_i}}$ holds for every $i$ and for some $t$ coprime to $\ell$, in other words if $y_I$ and $Y_I$ generate the same cyclic subgroup of $G$ of order $\ell^e$. 
By Pontryagin duality the number of those subgroups is the same as the number of cyclic quotients of $G$ of order $\ell^e$ and hence (i) follows.
For the first assertion of (ii) it suffices to choose the isomorphism of $G$ respecting the decomposition $\Gal(\mathbb Q(\zeta_n)/K')=\Gal(\mathbb Q(\zeta_{n_1})/K')\times \Gal(\mathbb Q(\zeta_{n_2})/\mathbb Q)$, where $n=n_1n_2$ and $n_2$ consists of all prime factors $q\nmid \ell\Omega$. We then conclude by a result on valuations \cite[Lemma C.1.7 and its proof]{hindry-silverman}.
\end{proof}

\begin{example}
As an aside remark, notice that $L$ may not be contained in the compositum of  $K$,  $\mathbb Q(\zeta_{\ell^\infty})$, and for any prime $q$ subextensions of $\mathbb Q(\zeta_q)/K\cap \mathbb Q(\zeta_q)$ of degree dividing $\ell^e$: consider $\ell^e=2$, $K=\mathbb Q(\sqrt{65})$, and the quartic cyclic field $L=\mathbb Q(\sqrt{65+8\sqrt{65}})\subset \mathbb Q(\zeta_{65})$.
\end{example}

\section{The Master Theorem on Artin type problems}

For $h_I\in \mathbb Z_{>0}^n$, $h:=\lcm(h_1,\ldots, h_n)$, and a positive squarefree integer $Q$, we write  
$$K_Q := K(\zeta_{Qh}, W_1^{1/Qh_1}, \ldots , W_n^{1/Qh_n})$$
and we let $C_Q \subset \Gal\big(FK_Q/K_1\big)$ consist of the automorphisms whose restriction to $F$ lies in $C$, and which, for any choice of $q\mid Q$ and $i\in I$, is not the identity on $K(\zeta_{qh_i}, W_i^{1/qh_i})$.
For every set $\mathcal S$ of primes of $K$ we write $\mathcal S(x):=| \{\mathfrak{p}\in \mathcal S : N\mathfrak{p} \le x\}|$ and 
$$\pi_K(x) := |\{\mathfrak{p} : N\mathfrak{p} \le x\}|= x/\log x + o(x/\log x)\,.$$ 
The density of $\mathcal{S}$ is defined by
$$d(\mathcal{S}) = \lim_{x \to \infty} \frac{\mathcal{S}(x)}{\pi_K(x)},$$
where it is possible that the limit does not exist.

\begin{theorem}\label{MasterTheorem}
Assume GRH.
Fix $h_I\in \mathbb Z_{>0}^n$, and let $S$ be the set of primes $\mathfrak p$ in the domain of $\Psi$ (restricting $\Psi$ to the primes of $K$ that are not ramified over $\mathbb Q$, do not ramify in $F$, and do not lie over a prime divisor of $\tau$) for which
\begin{align*}
\left(\frac{F/K}{\mathfrak{p}}\right) \subset C \quad \text{and} \quad \Psi(\p)=h_I\,.
\end{align*}
The set $S$ has a natural density $d(S)$, which is $0$ if and only if $S=\emptyset$ if and only if $C_Q = \emptyset$ for some positive squarefree integer $Q$.
Writing $Q_t:=\prod_{q\leqslant t} q$, we have
 $$d(S) = \lim_{t \to \infty} d_{Q_t}\,,\quad \text{where}\quad d_{Q_t}:=\frac{|C_{Q_t}|}{[FK_{Q_t} : K]}\,.$$
\end{theorem}

The property $\Psi(\p)=h_I$ means that $\mathfrak{p}$ splits in $K(\zeta_{h_i}, W_i^{1/h_i})$ and it does not split in $K(\zeta_{qh_i}, W_i^{1/qh_i})$, for any choice of $i\in I$ and of a prime $q$. Requiring the above conditions only for $q\leqslant t$, we obtain a larger set with natural density $d_{Q_t}$. Theorem~\ref{MasterTheorem} says that these larger sets are good approximations of $S$ for $t$ large.

\begin{proof}
For $Q$ a squarefree positive integer, we have $S\subset S_Q$, where 
$$S_Q:=\Big\{\mathfrak p \;: \!\! \quad \left(\frac{FK_Q/K}{\mathfrak{p}}\right) \subset C_Q \Big\}\,.$$
By the Chebotarev density theorem the natural density of $S_Q$ is $d_Q:={|C_{Q}|}/{[FK_{Q} : K]}$. We have $d_Q=0$ if and only if $C_Q=\emptyset$, and these conditions imply $S=\emptyset$.

We claim that for any positive integer $t$ we have
\begin{equation}
\label{eq:1}
S_{Q_t}(x) -S(x) \leqslant O(\pi_K(x)/t) + o(x/\log x)\,,
\end{equation}
where the implied constants depend on $n, K, F, W, h_I$. We deduce that $d(S)$ exists and (as $t\mapsto d_{Q_t}$ is non-increasing) it equals $\inf \{d_{Q_t}\}$.

Let $Q_0$ be the squarefree part of the constant $c$ from Proposition~\ref{justW} (iii), assuming $n! h\mid c$ and the additional condition involving $F$. We suppose  $C_{Q_0} \neq \emptyset$ and prove $\inf \{d_{Q_t}\}>0$. Let $L_q:=K(\zeta_q, W^{1/q})$ and $L_{q,i}:=K(\zeta_q, W_i^{1/q})$. Since $[L_{q,i}:K]\geqslant (q-1)q$ and by linear disjointness, for $t\geqslant Q_0$ we have 
\begin{align*}
d_{Q_t} = d_{Q_0} \prod_{q\leqslant t, q \nmid Q_0} \frac{|\Gal(L_q/K)\setminus \bigcup_{i\in I} \Gal\left(L_q/L_{q,i})\right)|}{|\Gal(L_q/K)|} \geqslant d_{Q_0} 
\prod_{q > n}\left(1 - \frac{n}{q(q-1)}\right).
\end{align*}
The product over primes $q > n$ is strictly positive: It suffices to show that the product over $n < q < y$ is bounded from below independently of $y$. Taking the logarithm of the partial products and using the bound $\log(1 - \alpha) \geqslant -2\alpha$ (valid for $0 \leqslant \alpha \leqslant 1/2$) we have
\begin{align*}
\prod_{n < q < y} \left(1 - \frac{n}{q(q-1)}\right) = \exp\left(\sum_{n < q < y} \log\left(1 - \frac{n}{q(q-1)}\right)\right) \geqslant \exp\left(-\sum_{n < q < y} \frac{2n}{q(q-1)}\right),
\end{align*}
which is bounded from below by $\exp(-2n\zeta(2)) = \exp(-\pi^2 n/3) > 0$. Hence $d_{Q_t}$ is bounded from below by a positive constant independent of $t$, as $d_{Q_0} > 0$.

We are left to prove the claim \eqref{eq:1}. For $t$ large enough, every $q>t$ satisfies $q\nmid h$ and $[L_{q,i}: K] \geqslant (q-1)q$. 
Setting $f_1(x) := \sqrt{x}/(\log x)^{100}$ and $f_2(x) := \sqrt{x}(\log x)^{100}$, consider the intervals 
$$I_1:=(t, f_1(x))\qquad I_2:=[f_1(x), f_2(x)] \qquad I_3:=(f_2(x) , x]$$ 
and the corresponding sets
$$\Gamma_j:=\{\mathfrak{p} : N\mathfrak{p} \leqslant x \text{ and }\mathfrak{p} \text{ splits in } L_{q,i} \text{ for some $i\in I$ and $q\in I_j$}\}\,.$$
Any $\mathfrak p\in S_{Q_t}\setminus S$ splits in some $L_{q,i}$ with $q>t$, and $N\mathfrak{p} \leqslant x$ implies $q\leqslant x$. Thus $S_{Q_t}(x) -S(x) = O\left(\left|\Gamma_1\right| + \left|\Gamma_2\right|+ \left|\Gamma_3\right|\right)$.
We clearly have 
$$\left|\Gamma_j\right| \leqslant \sum_{i\in I} \sum_{q \in I_j} \left|\lbrace \mathfrak{p} : N\mathfrak{p} \le x \text{ and } \mathfrak{p} \text{ splits in } L_{q,i} \rbrace\right|\,.$$
By the effective Chebotarev density theorem under GRH \cite{serre, Lag-Od}  we get 
$$\left|\Gamma_1\right|  \leqslant \sum_{i\in I} \sum_{q \in I_1} \left(\frac{\pi_K(x)}{q(q-1)} + O\left(\sqrt{x}(\log x)^{10}\right)\right) = \\
O\left(\pi_K(x)/t\right) + o(x/\log x)\,.$$
Moreover, considering separately the primes of $K$ of degree greater than $1$ we have $$\left|\Gamma_2\right| \leqslant n \sum_{q \in I_2} \left|\lbrace \mathfrak{p} : N\mathfrak{p} \leqslant x, \mathfrak{p} \text{ is of degree } 1 \text{ and } \mathfrak{p} \text{ splits in } K\left(\zeta_q\right)\rbrace\right| + n[K : \mathbb{Q}]\sqrt{x}$$
and hence by the Brun-Titchmarsh inequality
\begin{align*}
\left|\Gamma_2\right| & \leqslant O\big( \sum_{q \in I_2} \left|\lbrace p  :  p \leqslant x \text{ and } {p}\equiv 1\!\!\!\! \pmod q \rbrace\right|\big)+ n[K : \mathbb{Q}]\sqrt{x}\\
&  \leqslant O\left(
 \sum_{q \in I_2}\frac{2\pi(x)}{q-1}\right)+o(x/\log x)=o(x/\log x)\,,\end{align*}
where the last equality follows e.g.\ by decomposing $I_2$ into dyadic intervals of the form $[y, 2y]$, and bounding the contribution of such intervals by the prime number theorem and the estimate $1/(q-1) = O(1/y)$ for $q \in [y, 2y]$.
 
Fixing some $w_i\in W_i\setminus K^\times_{\tors}$ we have
$$\Gamma_3 \subset  
\bigcup_{i\in I} \{\mathfrak{p} : N\mathfrak{p} \le x \text{ and }\mathfrak{p} \text{ splits in } K\left(\zeta_q, w_i^{1/q}\right) \text{ for some $q\in I_3$}\}.$$
Thus $\p\in \Gamma_3$ implies $w_i^{(N\mathfrak{p} - 1)/q} \equiv 1 \pmod{\mathfrak{p}}$ for some $i\in I$, $q \in I_3$. So we have 
$$\prod_{\p\in \Gamma_3}N\mathfrak{p}\mid \prod_{i\in I} N_{K/\mathbb Q}\left(\prod_{k \leqslant x/f_2(x)} (w_i^k - 1)\right) \leqslant 2^{O((x/f_2(x))^2)}$$
(the implied constant also depends on $w_i$) and hence $\left|\Gamma_3\right| = o(x/\log x)$.
\end{proof}

\begin{remark}\label{masterremark}
Assuming GRH, $d(S)$ also exists if we replace $h_I$ by some non-empty $H\subset \mathbb Z_{>0}^n$, and we have
\begin{align}
\label{eq:density_sum_formula}
d(S(H))=\sum_{h_I\in H} d(S(h_I))\,.
\end{align}
For the existence of $d(S(H))$, let $K_C$ denote the primes $\mathfrak{p}$ of $K$ with $\left(\frac{F/K}{\mathfrak{p}}\right) \subset C$, let $d^\pm$ denote the upper/lower density and let $B(r) := \{1, \ldots , r\}^n$. We have
\begin{align*}
d_-(\Psi^{-1}H \cap K_C) \geqslant 
\lim_{r \to \infty} \sum_{h \in H \cap B(r)} d_-(\Psi^{-1}\{h\} \cap K_C) =
\lim_{r \to \infty} \sum_{h \in H \cap B(r)} d^+(\Psi^{-1}\{h\} \cap K_C) \geqslant\\
d^+(\Psi^{-1} H \cap K_C) - \lim_{r \to \infty}  d^+(\Psi^{-1} (\mathbb Z_{>0}^n \setminus B(r))) = 
d^+(\Psi^{-1}H \cap K_C)
\end{align*}
by Theorem~\ref{MasterTheorem} and because $d^-(\Psi^{-1}B(r))\rightarrow 1$. The last claim holds since it is well-known that under GRH for $a\in \mathbb Q\setminus \{0, \pm 1\}$
the density of the primes $p$ such that $\Ind_p(a) \leqslant c$ exists and tends to $1$ as $c \to \infty$, and the same holds for number fields following Hooley's proof \cite[Theorem 4.1, especially (3)]{hooley}, see also \cite[Theorem 4]{erdos-murty} or \cite[Theorem 5.1]{Wagstaff}.

Finally, \eqref{eq:density_sum_formula} follows from the observations above because, as $r \to \infty$, we have 
\begin{align*}
d(S(H)) = \sum_{h \in H \cap B(r)} d(S(h)) + o(1)\,.
\end{align*}
\end{remark}

In the following result we work with $\Psi_T$ for any squarefree integer $T\geqslant 2$. An image for this map is a tuple of images of $\Psi_\ell$ for $\ell\mid T$ prime. To ease notation we describe an element of $\Psi_T$ as a matrix whose entries are non-negative integers $e_{i,\ell}$.

\begin{corollary}\label{cor}
Assume GRH for (i).
\begin{enumerate}
\item[(i)]  Let $h_I\in \mathbb Z_{>0}^n$ and $h:=\lcm(h_1,\ldots, h_n)$. We have $h_I\in \im(\Psi)$ if and only if for every positive squarefree integer $Q$ there is an automorphism in 
$$\Gal\big(K(\zeta_{Qh},W_1^{1/Qh_1}, \ldots, W_n^{1/Qh_n})/K(\zeta_{h},W_1^{1/h_1}, \ldots, W_n^{1/h_n})\big)$$
which for any choice of $i\in I$ and $q\mid Q$ prime is not the identity on $K(\zeta_{q h_i}, W_i^{1/q h_i})$.
\item[(ii)] Let $e_I\in \mathbb Z_{\geqslant 0}^n$ and $e:=\max(e_1,\ldots, e_n)$. We have $e_I\in \im(\Psi_\ell)$ if and only if there is an automorphism in 
$$\Gal\big(K(\zeta_{\ell^{e+1}},W_1^{1/\ell^{e_1+1}}, \ldots, W_n^{1/\ell^{e_n+1}})/K(\zeta_{\ell^{e}},W_1^{1/\ell^{e_1}}, \ldots, W_n^{1/\ell^{e_n}})\big)$$
which for any choice of $i\in I$ is not the identity on $K(\zeta_{\ell^{e_i+1}}, W_i^{1/\ell^{e_i+1}})$.
\item[(iii)] Let $T\geqslant 2$ be a squarefree integer and let $\ell$ denote its prime factors. Let $E:=(e_{i,\ell})\in \mat (\mathbb Z_{\geqslant 0})$, set $H_i:=\prod_\ell \ell^{e_{i,\ell}}$ and $H:=\lcm(H_1,\ldots, H_n)$. Then we have $E\in \im(\Psi_T)$ if and only if there is an automorphism in 
$$\Gal\big(K(\zeta_{HT},W_1^{1/H_1T}, \ldots, W_n^{1/H_nT})/K(\zeta_{H},W_1^{1/H_1}, \ldots, W_n^{1/H_n})\big)$$
which for any choice of $i\in I$ and $\ell\mid T$ prime is not the identity on $K(\zeta_{\ell H_i}, W_i^{1/\ell H_i})$.
\end{enumerate}
\end{corollary}
\begin{proof}
For (i) we can apply Theorem~\ref{MasterTheorem} with $F = K$ because $C_Q \neq \emptyset$ is equivalent to the existence of a suitable automorphism. We now prove (ii), the proof of (iii) being analogous. Set $L: = K\left(\zeta_{\ell^{e+1}}, W_1^{1/\ell^{e_1 + 1}}, \ldots , W_n^{1/\ell^{e_n + 1}}\right)$. If a suitable automorphism $\sigma$ exists, then by the Chebotarev density theorem there are infinitely many primes $\mathfrak{p}$ of $K$ (in the domain of $\Psi$ and not ramifying in $L$) such that $\sigma \in \left(\frac{L/K}{\mathfrak{p}}\right)$ and hence $\Psi_{\ell}(\mathfrak{p}) = e_I$. Else, for all $\mathfrak{p}$ in the domain of $\Psi_{\ell}$ not ramifying in $L/K$ and for all $\sigma\in \left(\frac{L/K}{\mathfrak{p}}\right)$, the automorphism $\sigma$ is not as in the statement: if it does not fix $\zeta_{\ell^e}$ and $W_i^{1/\ell^{e_i}}$, then the index of $W_i \bmod \mathfrak p$ is not divisible by $\ell^{e_i}$; else it fixes some of $K(\zeta_{\ell^{e_i + 1}}, W_i^{1/\ell^{e_i + 1}})$ hence the index of $W_i \bmod \mathfrak p$ is divisible by $\ell^{e_i+1}$. We conclude because by Theorem~\ref{MasterTheorem} (see Remark~\ref{masterremark}) the non-empty preimages of $\Psi_\ell$ are sets of positive density. We may also conclude without assuming GRH because the primes $\mathfrak p$ ramifying in $K(\zeta_{\ell^\infty}, W^{1/\ell^\infty})/K$ are excluded from the domain of $\Psi_{\ell}$ (if $v_{\mathfrak p}(w)\neq 0$ for some $w\in W$, this is because $W \bmod \p$ is not well-defined; if $\mathfrak p$ is over $\ell$ and $\zeta_\ell\in K$, because $\ell$ ramifies in $K$) or $\mathfrak p$ is over $\ell$ and $\zeta_\ell\notin K$, which means that $\Psi_{\ell}(\mathfrak p)$ is the constant tuple $0$, and the preimage of this tuple is infinite.
\end{proof}

\begin{remark}\label{remarkNEW1}
In the Index Map Problem we assume that the groups have positive rank. If some of them is finite, then the preimage of any finite subset of $\mathbb Z^n_{>0}$  is finite, see Remark~\ref{remarkROOTS}. However, the preimage of an infinite set could be an infinite set of density $0$ or without a Dirichlet density. For example, consider $K=\mathbb Q$ and $W_1=\langle 1 \rangle$,   $W_2=\langle 2 \rangle$ and the preimage of $S\times \mathbb Z_{>0}$, where $S\subset \mathbb Z_{>0}$ is any infinite set such that $\{s+1 \mid s\in S\}$ is a set of primes of density $0$ (respectively, without a Dirichlet density).
\end{remark}

\section{Computability of the image of the index map}

For $J\subset I$ and $i\in I\setminus J$, we call $k_{J, i}$ the smallest positive integer $z$ such that 
$W_i^z \subset \prod_{j \in J} W_j$ (setting $k_{J, i}=\infty$ if there is no $z$) and we call $k:=\lcm(k_{J,i})$ (treating $\infty$ as $1$), where the least common multiple is over $J \subset I$ and $i \in I \setminus J$.

\begin{theorem}\label{computable-image}
Assume GRH. Let $h_I\in \mathbb Z_{>0}^n$, $h:=\lcm(h_1,\ldots, h_n)$, and let $\omega$ be the number of distinct prime divisors of $h$.
Assuming that arithmetic operations and computing prime factors can be performed in negligible time, checking whether $h_I\in \im (\Psi)$ can be done in $O(\omega)$ steps (the implied constant depends on $K, F, n, W_1,\ldots, W_n$).
\end{theorem}
\begin{proof}

\textbf{Step 1.} \emph{Reformulating the problem in terms of existence of a suitable automorphism, and reducing the problem to considering the Galois group corresponding to a single prime $q$.} By Theorem~\ref{MasterTheorem} and Proposition~\ref{justW}(iii), calling $Q$ the squarefree part of $c$, we have to check whether there is 
$$\sigma\in \Gal\left(F\left(\zeta_{Qh\tau}, W^{1/Qh}\right)/K\right)$$
such that $\sigma \mid_F\,\in C$ and $\sigma$ is the identity on $K(\zeta_{h_i}, W_i^{1/h_i})$ but for each choice of $i\in I$ and $q\mid Qh$ prime it is not the identity on $K(\zeta_{qh_i}, W_i^{1/qh_i})$.
Notice that, if $F'/F$ is a finite Galois extension and $C'\subset \Gal(F'/K)$ consists of the automorphisms whose restriction to $F$ is in $C$, then we do not change our problem by replacing $F$ by $F'$ and $C$ by $C'$.
Thus, up to extending $F$ independently of $h$ (and replacing $C$), by Proposition~\ref{justW}(ii) the cyclotomic-Kummer extensions for $W$ made with coprime parameters are linearly disjoint over $F$. Hence we may consider the problem separately for different values of $q \mid Qh$. 

\textbf{Step 2.} \emph{Reformulating the existence of a suitable automorphism in terms of generators of the groups $W_i$.} Call $e_I:=v_q(h_I)$. By symmetry we may assume $e_I$ is non-increasing. The splitting conditions then become: for all $i\in I$, 
\begin{equation}\label{sp-cond}
\text{$\sigma$ is the identity on $K(\zeta_{q^{e_i}}, W_i^{1/q^{e_i}})$ but not on $K(\zeta_{q^{e_i + 1}}, W_i^{1/q^{e_i + 1}})$.}
\end{equation}
We exclude the case where for some $j > i$ we have $W_j^{k}\subset \langle W_1,\ldots, W_{i}\rangle$ and $e_j <e_{i}-v_q(k\tau)$, as in this case we have $h_I \not\in \im \Psi$:  if $\sigma$ is the identity on $K(\zeta_{q^{e_{i'}}}, W_i^{1/q^{e_{i'}}})$ for all $i' \leqslant i$, then necessarily $\sigma$ is also the identity on $K(\zeta_{q^{e_j + 1}}, W_j^{1/q^{e_j + 1}})$.

We partition $I$ into intervals $T_j$, their starting points being $1$ and those $i\in I$ such that $e_i <e_{i-1}-v_q(k\tau)$. We call $\rho_i:=v_q(\# W_{i, tors})$, $\mathcal W_j:=\langle W_i : i\in T_j\rangle$, and for $i\in T_j$ we let $S_i\subset W_i$ be a minimal set (it is non-empty) of multiplicatively independent elements such that $W_i\subset \langle K^\times_{\tors}, S_i, \mathcal W_1^{1/k}, \ldots, \mathcal W_{j-1}^{1/k}\rangle$.

We prove that \eqref{sp-cond} holds for all $i \in I$ if and only if for all $i \in I$ we have 
\begin{equation}\label{sp-cond-new}
\text{$\sigma$ is the identity on $K(\zeta_{q^{e_i+\rho_i}}, S_i^{1/q^{e_i}})$ but not on $K(\zeta_{q^{e_i + \rho_i+ 1}}, S_i^{1/q^{e_i + 1}})$.}
\end{equation}
Notice that, for every $x\geqslant 1$, fixing $\mathbb{Q}(\zeta_{q^x}, W_i^{1/q^{x}})$ implies fixing $S_i^{1/q^{x}}$ and $\zeta_{q^{x + \rho_i}}$. Moreover, since $W_i \subset \langle K_{\tors}, S_i, \mathcal{W}_1^{1/k}, ... , \mathcal{W}_{j-1}^{1/k} \rangle$, then $W_i^{1/q^{y}}$ is fixed if $\zeta_{q^{y + \rho_i}}$, $S_i^{1/q^{y}}$, $\zeta_{q^{y+v_q(k\tau)-1}}$ and $\mathcal{W}_{j'}^{1/q^{y+v_q(k\tau)-1}}$ for $j' < j$ are all fixed. Taking $x = e_i$ and $y = e_i + 1$, the `only if' implication follows. The converse is proven similarly (by induction on $i$).

\textbf{Step 3.} \emph{Reformulating the existence of a suitable automorphism in terms of systems of linear (in)congruences.} Fixing $j$, let $e := e_{\min T_j}$, and let $\alpha_1, \ldots , \alpha_r\in \mathcal W_j$ be multiplicatively independent and such that $S_i\subset \langle K^\times_{\tors},\alpha_1, \ldots , \alpha_r\rangle$ for all $i\in T_j$. For any $\sigma$ define $\mathbf{x}:=(x_0,x_1,\ldots, x_r)$ as the smallest positive integers such that 
$$\sigma(\zeta_{Qh\tau}) = \zeta_{Qh\tau}^{x_0}\qquad\text{and}\quad \sigma\left(\alpha_i^{1/q^{e+1}}\right) = \zeta_{q^{e+1}}^{x_i} \alpha_i^{1/q^{e+1}}\quad\text{for all $i=1,\ldots, r$}\,.$$
For $i\in T_j$ and $w_{i,i'}\in S_i$ define $\mathbf{f}_{i, i'} = (f_0, \ldots , f_r)$ such that 
$$w_{i, i'}^{1/q^e} = \zeta_{q^{v_q(\tau)+e+1}}^{f_0}\alpha_i^{f_1/q^e} \cdots \alpha_r^{f_r/q^e}.$$
The splitting conditions concerning some $i\in T_j$ amount to

\begin{tabular}{lllll}
& $x_0 \equiv 1 \pmod{q^{e_i + \rho_i}}$ & {and} & $\langle \mathbf{x}, \mathbf{f}_{i, i'} \rangle \equiv 0 \pmod{q^{e_i}}$ &\,,\\
and $[$ & $x_0 \not\equiv 1 \pmod{q^{e_i + \rho_i+1}}$ & {or} & $\langle \mathbf{x}, \mathbf{f}_{i, i'} \rangle \not\equiv 0 \pmod{q^{e_i+1}}$ & {for some} $i'\;\;]$\,.
\end{tabular}

Here $\langle \cdot, \cdot \rangle$ is the usual inner product of vectors. 

\textbf{Step 4.} \emph{Showing how to efficiently decide whether the congruence problem has a solution.}

Considering all $i\in T_j$ we have various systems of linear incongruences, and the question is if at least one system has a solution in common with the linear congruences. Since there are only boundedly many system of incongruences, we may consider them separately and thus fix one of them.

Note that different sets $T_j$ give different systems of (in)congruences. The only common variable between these systems is $x_0$. Any common solution to the systems necessarily has $v_q(x_0 - 1) \geqslant \max_{i}(e_i + \rho_i)$, and thus we cannot use the incongruence for $x_0$ if $e_i + \rho_i$ is smaller than its maximal value over $i\in I$. Hence we essentially have no common variables between the different systems, and thus we may now fix some $T_j$.

Let $q^N$ and $S$ be the integer and the set from Lemma~\ref{lemma:kummer-galois}, and consider separately each $\mathbf{s}:=(s_0, \ldots, s_r)\in S$. Setting  $x_i := q^Ny_i + s_i$ with $\mathbf{y}:=(y_0,y_1,\ldots, y_r)$, and rewriting $\langle \mathbf{x}, \mathbf{f}_{i, i'} \rangle \equiv 0$ as $\langle \mathbf{y}, q^N\mathbf{f}_{i, i'}\rangle \equiv -\langle \mathbf{s}, \mathbf{f}_{i, i'}\rangle$, we get a system of the form
\begin{equation}
\label{system}
\begin{cases}
\langle \mathbf{y}, \mathbf{v}_i \rangle \equiv c(\mathbf{v}_i) \pmod{q^{e(\mathbf{v}_i)}} \quad \text{for all $i \in J$} \\
\langle \mathbf{y}, \mathbf{v}_i \rangle \not\equiv c(\mathbf{v}_i) \pmod{q^{e(\mathbf{v}_i) + 1}} \quad \text{for all $i \in J'$}
\end{cases}
\end{equation}
for some $J'\subset J\subset T_j$, for some vectors $\mathbf{v}_i$, and for some integers $c(\mathbf{v}_i)$ and $e(\mathbf{v}_i)\in \{e_i, e_i+\rho_i\}$. The vectors $\mathbf{v}_i$ consist of the vectors $q^N\mathbf{f}_{i, i'}$ and of the vector $(q^N, 0, \ldots , 0)$ to express the conditions for $x_0$. Notice that $\vert e(\mathbf{v}_i) - e(\mathbf{v}_{i'}) \vert \leqslant n v_q(k\tau)$ for all $i,i'\in J$.

Let $\mathcal{J}\subset J$ be a maximal subset such that $\mathbf{v}_i, i\in\mathcal{J}$ are $\mathbb Q$-linearly independent, and among the possible choices for $\mathcal {J}$, take any which minmizes $\sum_{i \in \mathcal{J}} e(\mathbf{v}_i)$. For $i\in J$ let $\mathbf{u}_i$ be the vector of coefficients to express $\mathbf{v}_i$ as a $\mathbb Q$-linear combination of $\mathbf{v}_j, j \in \mathcal{J}$. 

We claim that as $\mathbf{y}$ varies over integer vectors, the values for $\mathbf{z} :=  (\langle \mathbf{y}, \mathbf{v}_i \rangle)_{i \in \mathcal{J}}$ are all preimages of the (non-empty) set $\mathcal{M}$ of their reductions modulo $M$, for some positive integer $M$. Furthermore, $M$ and $\mathcal{M}$ can be calculated with an explicit finite procedure. Indeed, consider an integer tuple $(n_i)_{i \in \mathcal{J}}$ and the system
$$\langle \mathbf{y}, \mathbf{v}_i\rangle = n_i \quad \text{for all $i \in \mathcal{J}$}.$$
One can solve the system over $\mathbf{y} \in \mathbb{Q}^{|\mathcal{J}|}$, writing the components of $\mathbf{y}$ as $\mathbb{Q}$-linear combinations of $n_i$ (the coefficients depending on $\mathbf{v}_i$). Letting $M$ be the least common denominator of these coefficients, the components of $\mathbf{y}$ are integers if and only if certain linear combinations of $n_i$ are divisible by $M$. This is holds if and only if $(n_i)_{i \in \mathcal{J}}$ modulo ${M}$ lies in a certain set $\mathcal{M} \subset (\mathbb{Z}/M\mathbb{Z})^{|\mathcal{J}|}$, which may be computed by letting $\mathbf{y}$ range through $\{0, 1, \ldots , M-1\}^{|\mathcal{J}|}$ and considering the values of $(\langle \mathbf{y}, \mathbf{v}_i\rangle)_{i \in \mathcal{J}}$ modulo ${M}$.

Since $\langle \mathbf{y}, \mathbf{v}_i \rangle = \langle \mathbf{z}, \mathbf{u}_i \rangle$ we may rewrite \eqref{system} as
\begin{equation} \label{system2}
\begin{cases}
\langle \mathbf{z}, \mathbf{u}_i \rangle \equiv c(\mathbf{v}_i) \pmod{q^{e(\mathbf{v}_i)}} \quad \text{for all $i \in J$} \\
\langle \mathbf{z}, \mathbf{u}_i \rangle \not\equiv c(\mathbf{v}_i) \pmod{q^{e(\mathbf{v}_i) + 1}} \quad \text{for all $i \in J'$} \\
\left(\mathbf{z} \bmod{M}\right) \in \mathcal{M}.
\end{cases}
\end{equation}
By the Chinese remainder theorem we replace the last condition by $\left(\mathbf{z} \bmod{q^{v_q(M)}}\right) \in \left(\mathcal{M} \bmod{q^{v_q(M)}}\right)$ without changing the solvability of the system. As $\mathbf{u}_i, i \in \mathcal{J}$ are the standard basis for $\mathbb Q^{|\mathcal{J}|}$, $\mathbf{z}$ is determined modulo $q^{\min_{i \in \mathcal{J}} e(\mathbf{v}_i)}$.

Suppose first that $q$ is large enough (namely $q > |J'|$, $N = v_q(\tau) = v_q(M) = 0$, and 
the entries of $\mathbf{u}_i, i\in J$ are $0$ or have $q$-adic valuation $0$).
Then \eqref{system2} becomes
\begin{equation}\label{system3}
\begin{cases}
\langle \mathbf{z}, \mathbf{u}_i \rangle \equiv 0\pmod{q^{e(\mathbf{v}_i)}} \quad \text{for all $i \in J$} \\
\langle \mathbf{z}, \mathbf{u}_i \rangle \not\equiv 0 \pmod{q^{e(\mathbf{v}_i) + 1}} \quad \text{for all $i \in J'$},
\end{cases}
\end{equation}
As $\mathcal{J}$ is chosen such that $\sum_{i \in \mathcal{J}} e(\mathbf{v}_i)$ is minimized (among all $\mathcal{J}$ of maximal size), all congruences follow from the necessary conditions $z_i \equiv 0 \pmod{q^{e(\mathbf{v}_i)}}$ for all $i \in \mathcal J$. (Indeed, otherwise for some $i \in J$ and $i' \in \mathcal{J}$ we have $\mathbf{u}_{i, i'} \neq 0$ and $e(\mathbf{v}_i) < e(\mathbf{v}_{i'})$, and one could make $\sum_{i \in \mathcal{J}} e(\mathbf{v}_i)$ smaller by replacing $\mathcal{J}$ with $\mathcal{J} \cup \{i\} \setminus \{i'\}$.) As for the incongruences, there are $q^{|\mathcal{J}|}$ values of $\mathbf{z}$ where $0 \leqslant z_i < q^{e(\mathbf{v}_i) + 1}$ and $z_i \equiv 0 \pmod{q^{e(\mathbf{v}_i)}}$ for all $i\in \mathcal J$. If for some $i\in J'$ we have $\min_{j : \mathbf{u}_{i, j} \neq 0} e(\mathbf{v}_j) \geqslant e(\mathbf{v}_i) + 1$, then the incongruence for that index $i$ is not solvable, else the system is solvable, as each incongruence excludes at most $q^{|\mathcal{J}|-1}$ values.
All large values of $q$ may thus be solved at once (the $\mathbf{u}_i$'s depend on $q$ only through the ordering of the $e(\mathbf{v}_i)$'s).

For each of the finitely many remaining values of $q$ we may check the solvability of  \eqref{system2} by brute force. Indeed: Letting $V$ denote $\max_{i \in J} \max_{j \in \mathcal{J}} -v_q(\mathbf{u}_{i, j})$,  if $v_q(M) > \min_{i \in J} e(\mathbf{v}_i) - V$, this is because the moduli are bounded in terms of $K, F, n, W_1,\ldots, W_n$; else, we have $v_q(M) \leqslant \min_{i \in J} e(\mathbf{v}_i) - V$, and the congruences determine $(z_i \bmod q^{\max(e(\mathbf{v}_i) - V, 0)})$. Hence the condition on $(\mathbf{z} \bmod{q^{v_q(M)}})$ is either trivially satisfied or impossible, moreover the number of possible values for $(z_i \bmod q^{\max e(\mathbf{v}_i) + V + 1})$ to be checked is at most $q^{\max e(\mathbf{v}_i) + V + 1 - \min e(\mathbf{v}_i)}$.
\end{proof}

\begin{remark}\label{saving}
Assume GRH. Let $e_I\in \mathbb Z_{\geqslant 0}^n$ and consider the question whether $\Gal(\bar{K}/K)$ contains an automorphism whose restriction to $F$ is in $C$ and that for every $i\in I$ it is the identity on $K(\zeta_{q^{e_i}}, W_i^{1/q^{e_i}})$ but not on $K(\zeta_{q^{e_i}}, W_i^{1/q^{e_i+1}})$.
\begin{enumerate}
\item[(i)] By the proof of Theorem~\ref{computable-image},  the answer does not depend on $q$ for all $q$ larger than a constant (depending  on $K, F, n, W_1, \ldots , W_n$ and computable with an explicit finite procedure).
\item[(ii)] By the proof of Theorem~\ref{computable-image}, there is a constant $c$ (depending on $K$, $F$, $n$, $W_1, \ldots , W_n$ and computable with an explicit finite procedure) such that for any $q$ the answer is not affected by changing $e_I$ as follows: for some non-empty $J \subset I$ such that
$$d_J := \min_{j \in J} e_j - (c + \max_{i \in I \setminus J} e_i) > 0$$
we replace $e_i$ by $e_i-d_J$ for all $i\in J$. 
Indeed, if $c > v_q(k \tau)$, then the choice of the intervals $T_j$ in the proof stays the same, and if further $c > v_q(M) + V$ for any $M$ and $V$ as in the proof, we are either in the case of ``large'' $q$ or the case $v_q(M) \leqslant \min_{i \in J} e(\mathbf{v}_i) - V$ hence whether the system is solvable depends on the differences $e(\mathbf{v}_i) - e(\mathbf{v}_j)$ but not on the specific values of $e(\mathbf{v}_i)$.
\item[(iii)] If $F=K$, then we can answer by applying the Chebotarev density theorem (and the inclusion-exclusion principle), computing the degrees of the given  cyclotomic-Kummer extensions and of their various compositum fields.
\end{enumerate}
\end{remark}

\begin{example}
Assume GRH. Consider $K = F = \mathbb{Q}(i)$, $C = \{\text{id}_K\}$, $\alpha=2 + i$, $\beta = 3+2i$, and 
$$W_1 = \langle \alpha \rangle, \quad W_2 = \langle \beta \rangle, \quad W_3 = \langle \alpha\beta \rangle, \quad W_4 = \langle \alpha^2\beta \rangle\,.$$
By Example \ref{Gaussian} for every $m\geqslant 1$ we have $[K(\zeta_{\infty}, \alpha^{1/m}, \beta^{1/m}):K(\zeta_{\infty})]=m^2$. Hence for $h_I \in \mathbb{Z}_{> 0}^4$ we have $h_I \in \im(\Psi)$ if and only if $v_q(h_I) \in \im(\Psi_q)$ for every $q$. Indeed, by Theorem~\ref{MasterTheorem} the former condition means  $C_Q \neq \emptyset$ for all squarefree $Q$, and this in turn means (by the maximality of the above Kummer extensions) that $C_q \neq \emptyset$ for all primes $q$, i.e.\ $v_q(h_I) \in \im(\Psi_q)$ for all $q$.

For $q\neq 2$, $\im(\Psi_q)$ consists of those $e_I\in \mathbb{Z}_{\geqslant 0}^4$ such that $e_1 = e_2 = e_3 = e_4$ or there is a permutation $f$ such that $e_{f(1)} > e_{f(2)} = e_{f(3)}= e_{f(4)}$. We sketch the proof of this fact, supposing for simplicity that $e_I$ is non-increasing. Set $e:=e_1$. We use Corollary~\ref{cor} to reduce the question of whether $e_I \in \im(\Psi_q)$ to finding an automorphism $\sigma$ of $L := K(\zeta_{q^{e + 1}}, W^{1/q^{e + 1}})/K$ that for every $i$ is the identity on $K(\zeta_{q^{e_i}}, W_i^{1/q^{e_i}})$ but not on $K(\zeta_{q^{e_i+1}}, W_i^{1/q^{e_i+1}})$. Since we have $[L : K] = \varphi(q^{e+1})q^{2(e+1)}$, for any $\mathbf{x}=(x_0, x_1, x_2)$ such that $q\nmid x_0$ there exists an automorphism $\sigma$ of $L$ such that 
$$\sigma(\zeta_{q^{e+1}}) = \zeta_{q^{e+1}}^{x_0}, \quad \sigma\left(\alpha^{1/q^{e+1}}\right) = \zeta_{q^{e+1}}^{x_1}\alpha^{1/q^{e+1}}, \quad \sigma\left(\beta^{1/q^{e+1}}\right) = \zeta_{q^{e+1}}^{x_2}\beta^{1/q^{e+1}}\,.$$

We now look for $\mathbf{x}$ satisfying the congruences 
\[
\begin{cases}
x_0 \equiv 1 & \pmod{q^{e_1}} \\
x_1 \equiv 0 &\pmod{q^{e_1}} \\
x_2 \equiv 0 & \pmod{q^{e_2}} \\
x_1 + x_2 \equiv 0 & \pmod{q^{e_3}} \\
2x_1 + x_2 \equiv 0 & \pmod{q^{e_4}}
\end{cases}
\]
and some incongruences, e.g.\ for $W_4$ we can have $x_0 \not\equiv 1 \pmod{q^{e_4 + 1}}$ or $2x_1 + x_1 \not\equiv 0 \pmod{q^{e_4 + 1}}$. 
Partitioning $I$ into intervals regrouping the indices $i$ with the same $e_i$, we get either $\{1\}, \{2,3,4\}$ or $\{1,2,3,4\}$, else there is no $\mathbf{x}$ as requested.
In the former case there are solutions, as we require $x_0 \equiv 1 \pmod{q^{e_1}}, x_1 \equiv 0 \pmod{q^{e_1}}, x_2 \equiv 0 \pmod{q^{e_2}}$ and $x_2 \not\equiv 0 \pmod{q^{e_2+1}}$, and $x_0 \not\equiv 1 \pmod{q^{e_1+1}}$ or $x_1 \not\equiv 0 \pmod{q^{e_1+1}}$. In the latter case there are solutions, as we require $x_0 \equiv 1 \pmod{q^{e_1}}, x_1 \equiv x_2 \equiv 0 \pmod{q^{e_1}}$, and $x_0 \not\equiv 1 \pmod{q^{e_1+1}}$ (which can be satisfied for any $q \geqslant 3$) or $x_1, x_2, x_1+x_2, 2x_1 + x_2 \not\equiv 0 \pmod{q^{e_1+1}}$ (which can be satisfied for $q > 3$). \end{example}

\begin{theorem}\label{wishlist}
Assume GRH. The image of the index map is computable with an explicit finite procedure. Moreover, there exists a positive squarefree integer $Q$ (computable with an explicit finite procedure) such that the following holds: 
\begin{enumerate}
\item[(i)] The image of $\Psi_\ell$ for $\ell\nmid Q$ does not depend on $\ell$. 
\item[(ii)] For $h_I\in \mathbb Z_{>0}^n$, we have $h_I\in \im\Psi$ if and only if $v_Q(h_I)\in \im \Psi_Q$ and $v_\ell(h_I)\in \im \Psi_\ell$ for all $\ell\nmid Q$.
\end{enumerate}
\end{theorem}
\begin{proof}
Assertion (i) is a consequence of Remark~\ref{saving}(i) (and mimicking the proof of Theorem~\ref{computable-image} directly in the $\ell$-adic setting is likely to provide an alternative proof without assuming GRH). Indeed, as in the proof of Theorem~\ref{computable-image}, we can set $Q$ to be the squarefree part of the constant from Proposition~\ref{justW}(iii), where we require the additional condition with $F$ (enlarging $F$ as in the proof of Theorem~\ref{computable-image} and replacing $C$ accordingly).

We now prove (ii). Corollary~\ref{cor}(i) allows to reduce the condition $h_I\in \im\Psi$ to the existence of a certain automorphism. By the independence of Kummer-type extensions (see Proposition~\ref{justW}) we may reduce the problem to extensions w.r.t.\ $\ell \mid Q$ and other values of $\ell$ separately. Then we conclude by Corollary~\ref{cor}(ii)-(iii).

For the computability of the image, we are left to show that we can determine $\im \Psi_Q$ and $\im \Psi_\ell$ for some $\ell\nmid Q$. Notice that checking whether a given value is in $\im \Psi_Q$ (respectively, $\im \Psi_\ell$) can be done explicitly by Corollary~\ref{cor}.
We conclude because Remark~\ref{saving}(ii) reduces computing $\Psi_\ell$ to checking a computable finite list of values: to determine $\im \Psi_Q$, notice that in the proof of Theorem~\ref{computable-image} we extend $F$ so that the problems with different $\ell$ are independent (absorbing the interactions into the condition on the Artin symbol), thus primes $\ell\mid Q$ can be checked individually.
\end{proof}

\section{The index map for separated groups}\label{section-sep}

\begin{remark}\label{check}
Assume GRH, and recall Corollary~\ref{cor}. If the $W_i$'s have separated Kummer extensions w.r.t.\ $x_I\in \mathbb Z_{>0}^n$, then for all $y_I\in \mathbb Z_{>0}^n$ such that $x_I\in \otimes y_I$ we have $y_I\in \im (\Psi)$ if and only if there is an automorphism in
$$\Gal\big(K(\zeta_{\lcm(x_i)},W_1^{1/x_1}, \ldots, W_n^{1/x_n})/K(\zeta_{\lcm(y_i)},W_1^{1/y_1}, \ldots, W_n^{1/y_n})\big)$$
which, for every prime $q$ and $i\in I$ such that $q y_i\mid x_i$, is not the identity on $K(\zeta_{q y_i}, W_i^{1/q y_i})$.
Indeed, the difference w.r.t.\ Corollary~\ref{cor} (i) is considering the last condition only for  certain pairs $(q,i)$. This will be sufficient because those pairs $(q,i)$ such that $q y_i\mid x_i$ are taken care of by  assumption, and for the remaining ones we may perform an extension argument: in case $q y_i\nmid x_i$, we can consider the Kummer extension with parameter $qx_i$ in place of $x_i$ and by separatedness we can extend the given automorphism so that the requested condition also holds for this pair $(q,i)$. Similarly, if the $W_i$'s have $\ell$-separated Kummer extensions w.r.t.\ $x_I\in \mathbb Z_{\geqslant 0}^n$ and $y_I\in \mathbb Z_{\geqslant 0}^n$ is such that $x_I\in \oplus y_I$, then we have $y_I\in \im(\Psi_\ell)$ if and only if there is an automorphism in 
$$\Gal\big(K(\zeta_{\ell^{\max (x_i)}},W_1^{1/\ell^{x_1}}, \ldots, W_n^{1/\ell^{x_n}})/K(\zeta_{\ell^{\max (y_i)}},W_1^{1/\ell^{y_1}}, \ldots, W_n^{1/\ell^{y_n}})\big)$$
which, for every $i\in I$ such that $y_i< x_i$, is not the identity on $K(\zeta_{\ell^{y_i+1}}, W_i^{1/\ell^{y_i+1}})$.
\end{remark}

\begin{theorem}\label{thm-imf-sep}
Assume GRH. The following are equivalent for $x_I\in \mathbb Z_{>0}^n$ (respectively, for $x_I\in \mathbb Z_{\geqslant 0}^n$):
\begin{enumerate}
\item[(i)] The $W_i$'s have separated (respectively, $\ell$-separated) Kummer extensions w.r.t.\ $x_I$.
\item[(ii)] We have $\otimes x_I \subset \im (\Psi)$ (respectively, $\oplus x_I \subset \im (\Psi_\ell)$).
\item[(iii)] For each $y_I \in \mathbb Z_{>0}^n$ (respectively, $\mathbb Z_{\geqslant 0}^n$) the following holds: 
$$y_I \in \im(\Psi) \Leftrightarrow \gcd(y_I, x_I)\in \im(\Psi)$$ (respectively, $y_I\in \im(\Psi_\ell)  \Leftrightarrow \min(y_I, x_I)\in \im(\Psi_\ell)$).
\end{enumerate}
\end{theorem}
\begin{proof} 
For any $X_I\in \mathbb Z_{>0}^n$ such that the Kummer extensions are separated w.r.t.\ $X_I$ we have $X_I\in \im (\Psi)$, see the above remark. Thus (i) implies (ii) because $\otimes x_I \subset Sep$ by Proposition~\ref{used}. Reasoning similarly to the previous remark, (i) implies (iii). For every $x_I$ we have $\im(\Psi)\cap \otimes x_I\neq \emptyset$ (consider primes of $K$ that split completely in $K(\zeta_{\lcm(x_i)}, W^{1/\lcm(x_i)})$) and hence (iii) implies (ii). If (i) does not hold, then w.l.o.g.\ there are positive integers $N_1,N$ and a prime $q$ such that $x_1\mid N_1$, $\lcm (x_1,\ldots, x_n, q N_1)\mid N$ and  
$$K(\zeta_N, W_1^{1/q N_1}, W_{\neq 1}^{1/N})=K(\zeta_N, W_1^{1/N_1}, W_{\neq 1}^{1/N})$$ 
hence $(N_1, N,\ldots, N)\notin \im(\Psi)$ and (ii) does not hold. The equivalence of the respective assertions is analogous.
\end{proof}

\begin{remark}
To determine the image of $\Psi$ (respectively, $\Psi_\ell$) for separated groups, it suffices to compute the image of $\mathfrak p$ for the finitely many excluded primes (see Section~\ref{Notation}) and, by Proposition~\ref{used}(iii) and Theorem~\ref{thm-imf-sep}, apply finitely many times Remark~\ref{check}.
Moreover, $\im(\Psi)$ is the preimage of its reduction modulo $M$ for some positive integer $M$ if and only if the groups are separated, see (iii) and (ii) for $\ell\nmid M$ of Theorem~\ref{thm-imf-sep}.
\end{remark}

We then determine the image of the index map $p \mapsto \Ind_p(a)$ over $\mathbb{Q}$. Notice that we can write $a\in \mathbb Q^\times \setminus \{\pm 1\}$ uniquely in the form
\begin{align}
\label{eq:a-rep}
a = (-1)^\epsilon  \left( b^2 \cdot 2^{\delta} \cdot T\right)^{2^d}
\end{align}
where $\epsilon, \delta\in \{0,1\}$, $b \in \mathbb{Q}^{\times}$, $d\geqslant 0$, with at least one of $\delta = 1$ and $T \neq 1$ holding, and where $T$ is a squarefree integer of the form
$$T=\prod_{p_i \equiv 1 \!\!\!\!\pmod{4}} p_i \prod_{q_j \equiv 3\!\!\!\! \pmod{4}} (-q_j)$$
where $p_i, q_j$ are prime numbers. 

In the following result we see sets of excluded values for the index map: $E_{\square}$ is because the index for squares cannot be odd; $E_{T}$ is because in the given case $T\mid \Ind_p(a)$ implies $2\mid \Ind_p(a)$; 
$E_{2T}$ is because in the given case $2T\mid \Ind_p(a)$ implies $4\mid \Ind_p(a)$; $E_{3,T}$ is because in the given case $2^m(T/3) \mid \Ind_p(a)$ implies $3 \mid \Ind_p(a)$.

\begin{theorem}\label{image-f-Q} Let $a \in \mathbb{Q}^{\times} \setminus \{\pm 1\}$. The image of the index map $p \mapsto \Ind_p(a)$ as $p$ ranges over the odd primes with $v_p(a) = 0$ is 
$$\mathbb{Z}_{> 0} \setminus (E_{\square} \cup E_{T}\cup E_{2T} \cup E_{3, T})$$
where, writing $a$ as in \eqref{eq:a-rep}, we define 
\begin{center}
\begin{tabular}{rlll}
$E_{\square}$&$ =$&$ \{2n+1 : n \in \mathbb{Z}_{\geqslant 0}\}$ & \text{if $d \geqslant 1$ and $\epsilon = 0$}\\
$E_{T}$ &$=$&$ \{(2n+1)|T| : n \in \mathbb{Z}_{\geqslant 0}\}$ & \text{if $d = \delta = \epsilon = 0$}\\
$E_{2T}$ &$ =$&$ \{(2n+1)2|T| : n \in \mathbb{Z}_{\geqslant 0}\}$ & \text{if $d = \delta = \epsilon = 1$}\\
$E_{3, T}$ &$ =$&$ \{n2^m|T|/3 : n \in \mathbb{Z}_{\geqslant 0}, 3 \nmid n\}$  & \text{if $a$ is a cube in $\mathbb{Q}$ and $3 \mid T$}
\end{tabular}
\end{center}
(the sets have to be considered empty if the conditions are not met) and
$$m:=\begin{cases} 
d+1 & \text{if $\epsilon = \delta = 0$}\\
\max(d+1, 3) & \text{if $\epsilon = 0, \delta = 1$}\\
 d+2 & \text{if $\epsilon = 1, \delta = 0$}\\
 \max(d+2, 3) & \text{if $\epsilon = \delta = 1$ and $d \neq 1$}\\
 2 & \text{if $\epsilon = \delta = 1$ and $d = 1$}\,.\\
 \end{cases}$$
\end{theorem}

Notice that the image does not change by excluding a set of primes of density $0$ from the given index map.

\begin{remark}
After this work was accepted to publication, the authors noticed that Theorem \ref{image-f-Q} has previously been given by Lenstra \cite[(8.9) to (8.13)]{lenstra77} (though a detailed proof has not been provided there).
\end{remark}

\begin{proof} 
Let $D$ be the largest odd integer such that $a\in (\mathbb Q^\times)^D$. We set $E:=\max(3\delta, d+2)$ and $Z:=2^{E}\lcm(T,D)$. 
The group $\langle a \rangle$ has separated Kummer extensions w.r.t.~$Z$: following \cite{perucca-sgobba-tronto, PST-Q} we have to consider the $\ell$-adic failure (we have $\alpha^{1/{\ell^{x+1}}}\notin \mathbb Q(\zeta_{\ell^{\infty}}, \alpha^{1/{\ell^{x}}})$ for $\ell\neq 2$, $x\geqslant {v_\ell(D)}$ and for $\ell=2$, $x\geqslant E$) and the $2$-adelic failure (the related cyclotomic field is $\mathbb Q(\zeta_{2^E|T|})$).

By Theorem~\ref{thm-imf-sep} we only have to determine the positive divisors $z$ of $Z$ that are values for the index map, namely such that $\Gal(\mathbb{Q}(\zeta_{Z}, a^{1/Z})/\mathbb{Q}(\zeta_{z}, a^{1/z}))$ contains an automorphism which for every prime divisor $q$ of $Z/z$ is not the identity on $\mathbb{Q}(\zeta_{qz}, a^{1/qz})$. Fix a positive divisor $z$ of $Z$ and call $v:=v_2(z)$ and $V:=v_3(z)$. For a prime $q>3$ dividing $Z/z$ we have $\zeta_{zq}\notin \mathbb{Q}(\zeta_{Z/q}, a^{q/Z})$ so we are left to check whether $\Gal(\mathbb{Q}(\zeta_{6z}, a^{1/6z})/\mathbb{Q}(\zeta_{z}, a^{1/z}))$ contains an automorphism which is not the identity on $\mathbb{Q}(\zeta_{2z}, a^{1/2z})$ if $v<E$ and it is not the identity on $\mathbb{Q}(\zeta_{3z}, a^{1/3z})$ if $V< v_3(\lcm(T,D))$.
Such an automorphism exists unless $v<E$ and $\mathbb{Q}(\zeta_{2z}, a^{1/2z})=\mathbb{Q}(\zeta_{z}, a^{1/z})$, or $V<v_3(\lcm(T,D))$ and $\mathbb{Q}(\zeta_{3z}, a^{1/3z})=\mathbb{Q}(\zeta_{z}, a^{1/z})$.

We first determine when $\mathbb{Q}(\zeta_{2z}, a^{1/2z}) = \mathbb{Q}(\zeta_z, a^{1/z})$ holds. We will show that this is the case precisely when we are in one of the cases indicated by $E_{\square}, E_{T}$ and $E_{2T}$. We necessarily have $v \leqslant 2$ because $\zeta_{2^{v+1}} \not\in \mathbb{Q}(\zeta_{2^vw})$ for $w$ odd and $\mathbb{Q}(\zeta_z, a^{1/z})\cap \mathbb{Q}(\zeta_{2^\infty})$ is contained in $\mathbb{Q}(\zeta_{2^{\max(v+1,3)}})$.

If $v = 0$, the condition becomes $\sqrt{a} \in \mathbb{Q}(\zeta_z)$: this holds when either $d \geqslant 1$ and $\epsilon = 0$, or $d = \delta = \epsilon = 0$ and $T \mid z$. These obstructions correspond to the exceptional sets $E_{\square}$ and $E_{T}$.

If $v = 1$, the condition becomes $\mathbb{Q}(\zeta_{4}, a^{1/4}) \subset \mathbb{Q}(\zeta_{w}, a^{1/2})$ where $w:= z/2$ is odd. The former field must be abelian over $\mathbb Q$ hence $d \geqslant 1$. So the latter field contains $\zeta_4$ only if $\epsilon = 1$. Finally, from $a^{1/4} \in \mathbb{Q}(\zeta_{4w})$ we obtain $d = 1$, $\delta = 1$, $T \mid z$. This results in the set $E_{2T}$.

If $v = 2$, the condition becomes $\mathbb{Q}(\zeta_{8}, a^{1/8}) \subset \mathbb{Q}(\zeta_{4w}, a^{1/4})$ where $w:= z/4$ is odd. 
If $\epsilon = 0$, to get $\zeta_8$ we need $d \leqslant 1$, and to get $a^{1/8}$ we need $d \geqslant 2$, a contradiction.
If $\epsilon = 1$ we similarly have $d \geqslant 2$ and hence $\zeta_{16} \in \mathbb{Q}(\zeta_{8w}, a^{1/8})$, contradicting $\zeta_{16} \not\in \mathbb{Q}(\zeta_{4w}, a^{1/4})$. Hence, no obstructions arise from $v = 2$.

We now determine when $\mathbb{Q}(\zeta_{3z}, a^{1/3z}) = \mathbb{Q}(\zeta_{z}, a^{1/z})$. This will result in the exceptional set $E_{3,T}$. By considering the largest abelian subextensions, we have $V = 0$, $a$ is a cube, and $\zeta_3 \in \mathbb{Q}(\zeta_z, a^{1/2^v})$. The last condition holds if and only if $v \geqslant d+1$, $3\mid T$, $(T/3) \mid z$ and $\zeta_3 \in L(a^{1/2^{d+1}})$, where $L$ is the maximal subextension of $\mathbb{Q}(\zeta_z)$ of exponent $2$. So we must have $L(\zeta_3) = L(a^{1/2^{d+1}}) = L(\zeta_{2^{d+2}}^{\epsilon}\sqrt{(-3) \cdot 2^{\delta}})$.

For $\epsilon = \delta = 0$, no additional conditions are necessary.
For $\epsilon = 0, \delta = 1$, we need $v \geqslant 3$ to ensure $L(\sqrt{-3}) = L(\sqrt{-6})$. For $\epsilon = 1, \delta = 0$ we similarly need $v \geqslant d+2$.

Finally, for $\epsilon = \delta = 1$ we need for $d = 1$ (respectively, $d\neq 1$) that $v \geqslant 2$ (respectively, $v \geqslant \max(d+2, 3)$). Indeed,  $L(\sqrt{-3}) = L(\zeta_{2^{d+2}}\sqrt{-6})$ holds if and only if $\zeta_{2^{d+2}}\sqrt{2} \in L$. We conclude because the smallest cyclotomic extension in which $\zeta_{2^{d+2}}\sqrt{2}$ lies is: $\mathbb{Q}(\zeta_8)$ for $d = 0$; $\mathbb{Q}(\zeta_4)$ for $d = 1$; $\mathbb{Q}(\zeta_{2^{d+2}})$ for $d \geqslant 2$.
\end{proof}

The strategy of the above proof should extend to a number field $K$ known very explicitly (e.g.\ a quadratic field), and we could replace $a$ by a finitely generated subgroup of $K^\times$.

\begin{example}
The index map $p \mapsto \Ind_p(2)$ is surjective by Theorem~\ref{image-f-Q}, and also by Theorem~\ref{thm-imf-sep}, as its image contains all positive divisors of $8$ (consider $p=3, 7, 113, 73$).

The image of the index map $p \mapsto \Ind_p(-100)$ consists of the positive integers not congruent to $10$ modulo ${20}$. These values must be excluded, as $-100$ is minus a square, thus $10 \mid \Ind_p(-100)$ implies $p \equiv 1 \pmod{20}$ and hence $20 \mid \Ind_p(-100)$ because $-100=(\sqrt{5}(1+\zeta_4))^4$. 

For $a=-3$ (respectively, $a=(-3)^3$) and $p\geqslant 5$ we have $3\mid \Ind_p(a)$ only if (respectively, if and only if) $2\mid \Ind_p(a)$ because $\mathbb Q(\zeta_2, a^{1/2})$ is contained in (respectively, equals) $\mathbb Q(\zeta_3, a^{1/3})$. 
\end{example}

\begin{remark}\label{remarkROOTS}
 Over $K$, fix some positive integer $m\mid \tau$ and consider the index map $\mathfrak p \mapsto \Ind_\p(\zeta_m)$ for $\mathfrak p\nmid mO_K$. Its image has density $0$ inside $\mathbb Z_{>0}$ because $\Ind_\p(\zeta_m)=(N\mathfrak p-1)/m$. For $K=\mathbb Q$ we have $\Ind_2(-1)=1$ and $\Ind_p(-1)=(p-1)/2$ for $p$ odd, while for $K=\mathbb Q(\zeta_4)$ we have $\Ind_{1+\zeta_4}(\zeta_4)=1$ and $\Ind_{\mathfrak p}(\zeta_4)$ is $(p-1)/2$ (respectively, $(p^2-1)/2$) if $\mathfrak p$ lies over a prime $p\equiv 1 \pmod 4$ (respectively, $p\equiv 3 \pmod 4$).
These examples show that the image of the index map for one or several finite groups is easy to describe, however the description is not explicit. Thus, if we remove from the Index Map Problem the assumption that the groups have positive rank, then we may not be able to compute the image of the index map with a finite procedure.
 \end{remark}

\section{Examples of groups with surjective index map}\label{surjective}

\begin{proposition} \label{infinitelymany}
If $K \neq \mathbb{Q}$, then there exists a sequence $(\alpha_i)_{i \in \mathbb{Z}_{> 0}}$ with $\alpha_i \in K^{\times}$ for all $i > 0$ which are algebraic integers and not units, whose norms $N(\alpha_i)$ are pairwise coprime, and such that for all positive integers $r,n$ we have
\begin{equation}\label{maxdeg}
[K(\zeta_\infty, \alpha_1^{1/n}, \ldots, \alpha_r^{1/n}): K(\zeta_\infty)]=n^r\,.
\end{equation}
Assuming GRH, for any distinct $i_1,\ldots, i_r$ the index map $\mathfrak{p} \to (\Ind_{\mathfrak{p}}(\alpha_{i_1}), \ldots , \Ind_{\mathfrak{p}}(\alpha_{i_r}))$ is surjective onto $\mathbb Z_{>0}^r$.
\end{proposition}
\begin{proof} By Theorem~\ref{thm-imf-sep} the second assertion follows from \eqref{maxdeg}, and it suffices to show \eqref{maxdeg} when $n=\ell$ is prime. 
Call $z$ the squarefree part of $2\Omega$. For all $\ell\mid z$ at once we construct by induction $\alpha_{\ell,i}$ as in Lemma~\ref{one-ell} and with pairwise coprime norms. The algebraic integers $\alpha_i:=\prod_{\ell\mid z} \alpha_{\ell, i}^{z/\ell}$ are not units and have coprime norms. Then \eqref{maxdeg} holds for $\ell\mid z$ because given integers $z_j$ not all divisible by $\ell$, no $\ell$-th root of $\prod_j \alpha_j^{z_j}$, equivalently of $\prod_j \alpha_{\ell,j}^{z_jz/\ell}$, is in $K(\zeta_\infty)$ because there are distinct primes $q_j$ such that there is no $t$ as in the statement of Lemma~\ref{one-ell} for $q:=q_j$ and $\alpha:=\alpha_{\ell, j}$.
It also holds for $\ell\nmid z$ by Proposition~\ref{gq}, Remark~\ref{Debry}, and \cite[Theorem 18]{debry-perucca}, as for every $i$ there is some prime $\mathfrak p$ of $K$ such that $\ell\nmid v_{\mathfrak p}(\alpha_i)$ and $\ell\mid v_{\mathfrak p}(\alpha_j)$ for all $j\neq i$.
\end{proof}

\begin{lemma}\label{one-ell}
Let $K \neq \mathbb{Q}$, fix some prime $\ell$ and $x>0$, and for each $q \equiv 1 \pmod{\ell}$ let $1 \leqslant e_q \leqslant v_\ell(q-1)$ be arbitrary and let $g_q$ be as in Proposition~\ref{gq}. There exists an algebraic integer $\alpha\in K^\times$ which is not a unit such that the following holds: $\alpha^{1/\ell}\notin K(\zeta_\infty)$; for all primes $\mathfrak p$ of $K$ we have $v_{\mathfrak p}(\alpha)<\ell$; the norm $N(\alpha)$ is only divisible by primes $q>x$ not inert in $K$ such that $q\nmid \ell\Omega$, $q\equiv 1 \pmod \ell$; for some of these $q$ there is no $t\in \mathbb Z$ such that $v_{\mathfrak q}(\alpha)\equiv t v_{\mathfrak q}(g^{\ell^{e_q}}_q) \pmod \ell$ holds for all primes $\mathfrak q$  of $K$ over $q$.
\end{lemma}
\begin{proof}
There are infinitely many $q>x$ not inert in $K$ such that $q\nmid \ell\Omega$, $q\equiv 1 \pmod \ell$. Considering the primes $\mathfrak q_j$ over $q$, by a counting argument there is an ideal $I_q:=\prod_j \mathfrak{q}_j^{z_j}$ such that  $0\leqslant z_j<\ell$ and for no $t\in \mathbb Z$ we have $z_j\equiv v_{\mathfrak{q}_j}(g^{t\ell^{e_q}}_q)\pmod \ell$ for all $j$. For some set $S$ the ideal $\prod_{q\in S} I_q$ is principal: calling $\alpha$ a generator, 
for $q\in S$ we have $v_{\mathfrak{q}_j}(\alpha)=z_j$ hence $\alpha^{1/\ell}\notin K(\zeta_\infty)$ by  Proposition~\ref{gq}.
\end{proof}

\begin{example}\label{Gaussian}
If $K=\mathbb Q(\alpha_1,\ldots, \alpha_r)$ and $\alpha_i^2$ are distinct rational primes greater than $3$, or if $K = \mathbb{Q}(i)$ and $\alpha_i$ are Gaussian primes whose norms are distinct primes (see Remark~\ref{Debry}), then \eqref{maxdeg} holds for every $n$.
\end{example}

\begin{remark}
We expect \eqref{maxdeg} to hold for generic elements $\alpha_1,\ldots, \alpha_r\in K^\times$. Firstly, we expect that the $\alpha_i$'s are strongly $\ell$-independent for every prime $\ell$ and strongly $2$-independent over $K(\zeta_4)$ (see \cite{debry-perucca}) hence $[K(\zeta_{\ell^\infty}, \alpha_1^{1/\ell^n}, \ldots, \alpha_r^{1/\ell^n}): K(\zeta_{\ell^\infty})]=\ell^{nr}$.
We are left to check that for $\ell\mid \tau$ and $\alpha:=\prod_i \alpha_i^{z_i}$ (where $0\leqslant z_i<\ell$ and the $z_i$'s are not all zero) we have $\alpha^{1/\ell} \notin K(\zeta_{\infty})$. 
If we had $\alpha^{1/\ell} \in K(\zeta_{\infty})$, then $\ell\mid v_\mathfrak p(\alpha)$ for all primes $\mathfrak p$ of $K$ over $p\nmid \tau$ and over $p\not \equiv 1 \pmod \ell$ (see Proposition~\ref{gq}), and there are few admissible values for the tuple $(v_{\mathfrak q}(\alpha))$, where $\mathfrak q$ varies over the primes of $K$ over any fixed $q \equiv 1 \pmod \ell$ not inert in $K$ such that $q\neq \ell\Omega$ (see the proof of Lemma~\ref{one-ell}).
Generically, the non-zero valuations of the $\alpha_i$'s are (with few exceptions) $1$ and correspond to primes of $K$ of degree $1$ lying over distinct primes $q$. For such $q$, if $z_i\neq 0$, then the tuple $(v_{\mathfrak{q}}(\alpha))$ for the primes $\mathfrak{q}$ over $q$ has one non-zero entry hence $\alpha^{1/\ell} \notin K(\zeta_{\infty})$ by \cite[Theorem 12]{HPST}.

 \end{remark}

\begin{example}
The index map is not surjective if $K=\mathbb Q$ and $n\geqslant 2$.
If $a,b\in \mathbb Q^\times$, then the index map $p\mapsto (\Ind_p(a), \Ind_p(b))$ is not surjective because, considering the fraction $a=\frac{n}{d}$, the condition $4 nd \mid \Ind_p(b)$ implies $p\equiv 1 \pmod{4|nd|}$ hence $2\mid \Ind_p(a)$. Similarly, if $W_1=\langle \frac{n_1}{d_1}, \ldots, \frac{n_r}{d_r}\rangle$, then we cannot have $2\nmid \Ind_p(W_1)$ and $4n_1d_1\cdots n_rd_r\mid \Ind_p(W_2)$.
\end{example}

\subsection*{Acknowledgements} We would like to thank Fritz Hörmann, Flavio Perissinotto, and Pietro Sgobba for helpful discussions. We also thank the two anonymous referee for their valuable work and their insights. The first author was supported by the Emil Aaltonen foundation and worked in the Finnish Centre of Excellence in Randomness and Structures (Academy of Finland grant no. 346307).

\bibliographystyle{plain}

{\tiny

\textbf{Conflict of Interest}

The authors have no conflicts of interest to disclose.

\textbf{Data Availability}

Data sharing not applicable to this article as no datasets were generated or analysed
during the current study.}

\end{document}